\documentclass[oneside]{amsart}
\usepackage{bbm}
\usepackage{amsmath,amssymb,amsthm}
\usepackage{mathrsfs}
\usepackage{color}
\usepackage{xcolor}
\usepackage{graphicx}
\usepackage{enumerate}
\usepackage{ifpdf}
\usepackage{booktabs}
\usepackage{algorithmic}
\usepackage{float}
\usepackage{CJK}
\usepackage{listings}
\usepackage{bbm}
\usepackage{mathrsfs}
\usepackage{amsmath}
\usepackage{epstopdf}
\usepackage{cases}
\usepackage{subfigure}
\usepackage{subcaption}
\usepackage{enumerate}
\usepackage{enumitem}
\usepackage{lipsum}
\usepackage{xcolor}
\usepackage{float}
\usepackage{colortbl}  
\usepackage{comment}
\usepackage{pgfplots}
\usepackage{adjustbox} 
\pgfplotsset{compat=newest}
\newtheorem{assumption}{Assumption}

\usepackage{diagbox}
\makeatletter
\def\ps@pprintTitle{%
 \let\@oddhead\@empty
 \let\@evenhead\@empty
 \def\@oddfoot{}%
 \let\@evenfoot\@oddfoot}
\makeatother
 \usepackage{graphicx}
\usepackage{amssymb}
\usepackage{amsthm}
\usepackage{cancel}

\renewcommand{\theequation}{\thesection.\arabic{equation}}
\makeatletter\@addtoreset{equation}{section} \makeatother
\newtheorem{theorem}{Theorem}[section]
\newtheorem{lemma}{Lemma}[section]

\newcommand{\R}{\mathbb{R}}

\begin{document}
\title{Quasi-interpolation with random sampling centers}

\author{Wenwu Gao}

\address{School of Big Data and Statistics,  Anhui University, Hefei, P. R. China}
\email{wenwugao528@163.com}
\thanks{This work is supported by  Youth Project (Class A) of Anhui Provincial Natural Science Foundation (No. 2508085J009), NSFC (12271002).}
\thanks{The third author is the corresponding author.}

\author{Le Hu}
\address{School of Big Data and Statistics,  Anhui University, Hefei, P. R. China}
\email{hule1907356451@163.com}

 \author{Xingping Sun}
\address{Department of Mathematics, Missouri State University, Springfield, MO 65897, USA}
\email{XSun@MissouriState.edu}

\author{Xuan Zhou}
 \address{Amazon Web Services, Shanghai, P. R. China}
 \email{zhouxss@amazon.com}

\subjclass[2020]{Primary 41A05, 41A63, 65D05, 65D10, 65D15}


\keywords{ Function approximation; Stochastic quasi-interpolation;  McDiarmid-type concentration inequality; Random sampling; Probabilistic convergence}

\begin{abstract}
We propose and study a  general  quasi-interpolation framework for stochastic function approximation, which stems and draws motivation from convolution-type solutions for certain practical weighted variational problems.
 We obtain our quasi-interpolants using Monte Carlo discretization of the pertinent integrals and establish a family of $L^p$-McDiarmid-type  concentration inequalities for $1\leq p\leq \infty$, which resulted in verifiable expected error estimates for the stochastic quasi-interpolants.   
The $L^1$-version of these concentration inequalities 
  is dynamically-independent of dimensions, which offers a partial stochastic  mitigation of the so called ``curse of dimensionality". The $L^\infty$-version of these concentration inequalities strengthens the existing expected $L^\infty$-error estimates in the literature. Numerical simulation results are provided at the end of the paper to validate the underlying theoretical analysis.
 \end{abstract}
\maketitle



\section{\textsl{Introduction}}
 Quasi-interpolation plays a vital role in function approximation \cite{buhmann2022quasi}, a fundamental problem  in all branches  of data science \cite{fasshauer2007meshfree-book},
  \cite{griebel2005sparse}, 
 \cite{lin_SINUM2021_distributed}, 
\cite{montufar2022distributed}. Using a quasi-interpolation method,
one can easily generate a sequence of approximants to an unknown target function directly from available data. This is in sharp contrast to many interpolation methods which often involve solving large-scale linear systems of equations \cite{fasshauer2007meshfree-book}. The versatility of quasi-interpolation manifests in many areas and appears in various formats. In particular, kernel based quasi-interpolation schemes enjoy a certain type of optimality and regularization properties  \cite{gao2020optimality}, which allows one to tailor-make kernels so that  the resulted  quasi-interpolants achieve some desirable theoretical and practical goals such as pre-designed locality and structure preservation   \cite{buhmann2022quasi}, 
\cite{gao-fasshauer-fisher2022divergence}, 
\cite{gao2018multiquadric},    
 \cite{wu1994shape}. These appealing features have driven the study and research of 
quasi-interpolation to a new height; see e.g. \cite{beatson1996multiquadric},  \cite{buhmann1990multivariate}, 
 \cite{buhmann2024compression}, 
\cite{buhmann1995quasi}, 
\cite{buhmann2024new}, \cite{buhmann2015pointwise}, \cite{fasshauer2007iterated},  \cite{franz-wendland2023multilevel}, \cite{grohs2013quasi}, \cite{Grohs2017scattered},
 \cite{hubbert2023convergence},  \cite{mazya-schimidt2001quasi},
\cite{ortmann2024high}.

In modeling many practical problems, choices of sampling centers are rarely at one's disposal. As such, many researchers have been studying  quasi-interpolation problems based on irregular sampling centers (see \cite{buhmann-2022JFAA-strict} and the references therein).  In the age of AI,   quasi-interpolation  modeling
with random sampling centers over compact domains or manifolds has been showing computational efficiency; see e.g. \cite{dai2024random}, \cite{gao2020multivariate},  \cite{gao2024quasi},  
\cite{krieg2025function},\cite{montufar2022distributed},  \cite{sommariva2025random}, \cite{sun2021probabilistic}, \cite{sun2022stochastic}, \cite{wu2013sampling}, \cite{wu2021probabilistic}, \cite{wu2020polynomial}.  

In this paper, we study and implement a new quasi-interpolation scheme which stems from the convolution-type solution of the following kernel based variational problem:
\begin{equation}\label{variationalprobleminte}
   \arg\min_{f_0\in C(\Omega)} \int_{\Omega}\int_{\Omega}\left[f(t)-f_0(x)\right]^2\psi_h(x-t)d\mu(t)dx.
\end{equation}
Here  $\Omega \subset \R^d$ is   a compact convex domain with Lipschitz boundary; $\mu$ is a probability measure on $\Omega;$ $f\in C(\Omega) $ is a target function;  $\psi_h:=h^{-d}\psi(h^{-1}(x-t)),\; h>0,$  where $\psi \in C(\R^d)$ is a nonnegative approximate identity in the sense that $\int_{\R^d} \psi(x)dx=1.$
Employing a standard technique in calculus of variation, we find the unique minimizer $f_{\rm min}$ of the functional in \eqref{variationalprobleminte} to be:
\begin{equation}\label{f-min}
 f_{\rm min}(x)  =\frac{\int_{\Omega}f(y)\psi_h(x-y)d\mu(y) }{\int_{\Omega}\psi_h(x-t)d\mu(t)}.
\end{equation} 
Here we assume that there is a constant $c>0$, such that
\[
\int_{\Omega}\psi_h(x-t)d\mu(t)>c, \quad x \in \Omega.
\]
Literature has been flourishing with innovative ways of discretizing the convolution integrals in $f_{\rm min}$ to form desirable quasi-interpolants. If $\Omega = \R^d$, and $\mu$ is the Lebesgue measure on $\R^d$, then the denominator on the right hand side of \eqref{f-min} identically equals to $1.$ 
For some appropriately chosen $\psi$,
discretizing the numerator based on the $h$-spaced grid $h\mathbb{Z}^d$ leads to the so called ``cardinal interpolation" advocated by Schoenberg  \cite{maz1996approximate}, \cite{schoenberg1946contributions}. Other methods followed in earnest. To name a few, we mention  sparse-grid  quasi-interpolation \cite{dung2016sampling}, \cite{dung2020dimension}, \cite{gao2024quasi}, \cite{jeong2021approximation},  \cite{usta-levesley2018multilevel},    Monte Carlo quasi-interpolation \cite{gao2020multivariate},   Quasi Monte Carlo quasi-interpolation \cite{gao2020optimality}, moving   least squares \cite{bos1989moving}, scaled Shepard's method \cite{senger2018optimal}, \cite{shepard1968two},  as well as many other classical quasi-interpolation schemes \cite{buhmann2022quasi}.  

In this paper,
we use Monte Carlo discretization  to generate our quasi-interpolants. To elaborate, let $X_1,\ldots,X_N$ be $N$ independent copies of the random vector $X$ with law $\mu,$ which satisfies Assumption \eqref{assumption:A} in Section 2.
We discretize the convolution integrals in both the numerator and the denominator of  \eqref{f-min}
independently and simultaneously, which gives rise to an quasi-interpolant $Q_h^rf$ in the following rational formation:
\begin{equation}\label{eq:Q_h^rf}
Q_h^rf(x):=\sum^N_{j=1}f(X_j)\frac{\psi_h(x- X_j)}{\sum_{\ell=1}^N\psi_h(x-X_{\ell})}, \ x\in \Omega.
\end{equation} 

Error analysis for our quasi-interpolant $Q_h^rf$ is done in two phases. Phase I. For a given target function $f \in C(\Omega)$, we derive in
Theorem \ref{th:convolution_error} a Jackson-type error estimate for
$\|f-f_{\rm min}\|_p$ in terms of the modulus of continuity $\omega_f$
of $f$; see \eqref{modulus-of-continuity} for the definition of $\omega_f$. The error analysis in Phase I is deterministic.
Phase II. We estimate quantitatively the $L^p$-probabilistic convergence of our quasi-interpolants to its target function and the expected upper bound for 
$\|Q_h^rf-f_{\rm min}\|_p$. 
Main results in the latter include a family of $L^p$-McDiarmid-type  concentration inequalities \cite{combes2024extension} (Theorems \ref{convergenceanalysi_infty} and \ref{convergenceanalysi_1}), which yield   the mean  $ L^p $-convergence orders of our quasi-interpolants (Theorem \ref{mean Lp eatimate}). 
As highlights of the main results of the current paper, we mention the following two inequalities:
$$
    \mathbb{P}\{\|Q_h^rf-f\|_p>\epsilon\}
    \lesssim 
     \exp\left\{-\frac{CNh^d\epsilon^2}{2(1+\epsilon)}\right\},$$ 
     for parameters $\epsilon, h,$ and $N$ in the range specified by:
 \begin{equation*}
   \omega_f(h)\lesssim \epsilon \  \text{and} \  N^{-\frac{1}{2}}h^{-\frac{d}{2}}|\mathrm{log} \ h|^{\frac{1}{2}}\lesssim \epsilon.
 \end{equation*}
    Here $C$ is a specifiable absolute constant. 
    The notation $A(N,h,\epsilon) \lesssim B(N,h,\epsilon)$  means  that there exists a positive constant $C$ independent of the parameters $N,h,\epsilon$,  
    such that  $A(N,h,\epsilon) \leq C\ B(N,h,\epsilon)$. 
For $1\le p\le 2$,  we obtain the following stronger inequality:
\begin{equation}\label{1p2}
 \mathbb{P}\{\|Q_h^rf-f\|_{p}>\epsilon\}
    \lesssim \exp\left\{-C Nh ^{2d\left(1-\frac{1}{p}\right)}\epsilon^2\right\}, 
\end{equation} 
     for parameters $\epsilon, h,$ and $N$ in the range specified by:
 \begin{equation}\label{tiao-jian}
   \omega_f(h)\lesssim \epsilon  \ \text{and} \ N^{-\frac{1}{2}}h^{-\frac{d}{2}} \lesssim \epsilon.
 \end{equation}
 If $p=1, $ then the right hand of \eqref{1p2} reduces to the clean sub-Gaussian decay form $\exp\left\{-C N \epsilon^2\right\}$,
 which means that the $L^1$-version of \eqref{1p2} 
 is independent of the dimension $d$ under the conditions of \eqref{tiao-jian}. 
 Therefore, it offers a partial-stochastic way of mitigating the curse of dimensionality. 
Based on these inequalities, we derive   the following mean $L^p$-convergence orders for an order $s\ (0<s\le 1)$ H\"{o}lder class function{\footnote{An order $s\ (0<s\le 1)$ H\"{o}lder class function $f$ can be equivalently characterized by the asymptotic equation: $\omega_f(h)=\mathcal{O}{(h^s)}$.}}:
\begin{equation*}
     \mathbb{E}\left\Vert Q_h^rf-f\right\Vert_{p}
    \lesssim \left\{\begin{array}{ll}
   N^{-\frac{s}{2s+d}} , & \hbox{if} \quad 1 \le p \le 2,\\
  N^{-\frac{s}{2s+d}}(\log N)^{\frac{1}{2}}  & \hbox{if} \quad 2<p \le \infty.
       \end{array} \right. 
\end{equation*}

The paper is organized as follows. Section $2$ lays out background information and
reviews basic definitions and assumptions which will be used in subsequent texts.  Section $3$ contains the statements of the main results and their proofs.
Section $4$ provides several numerical simulation results that validate our theoretical analysis.

\section{Mathematical preliminaries}
Denote by  $L^p(\Omega)$ the Banach spaces consisting of all Lebesgue measurable functions $f$  on $\Omega$  for which $\Vert f \Vert_{p}<\infty$.
Here as usual,  the $L^p$-norm  of $f$
is defined  by
\begin{equation*}
    \Vert f \Vert_{p} := \left\{
        \begin{array}{ll}
            \left( \int_{\Omega} |f(x)|^p \, dx \right)^{1/p} & \text{if } 1 \leq p < \infty, \\
            \inf \{ C : |f(x)| \leq C \text{ for almost all } x \} & \text{if } p = \infty.
        \end{array}
    \right.
\end{equation*}
For the sake of distinction and clarity, we denote
$\Vert f \Vert_{p,\mathbb{R}^d}$  the $L^p$-norm of $f \in L^p(\mathbb{R}^d)$.
  We define $\omega_f$,
  the modulus of continuity  of $f$ by:
  \begin{equation}
  \label{modulus-of-continuity}
  \omega_f(t)=\sup_{|x-y|\le t}|f(x)-f(y)|, \ x,y \in \Omega.
  \end{equation}
One can verify that $\omega_f(t)$ satisfies the following inequality often referred to in the literature as “sub-additivity”:
$$
\omega_f(\delta t)\le (1+\delta)\omega_f(t),\ \delta\geq 0,\ t\geq 0.
$$
  We adopt the following format for the Fourier transform pair:
  \begin{equation*}
\label{Fourier transform}
 \hat{f}(\xi)=\int_{\mathbb{R}^d}f(x)e^{-i\xi x}dx,\ \  f(x)=(2\pi)^{-d}\int_{\mathbb{R}^d}\hat{f}(\xi)e^{i\xi x}dx, \quad f \in L_1(\mathbb{R}^d),
 \end{equation*}
with the understanding that the Fourier transform and its inverse have been extended in a conventional manner so that they become applicable to Schwarz class distributions. 
   
 We take independent and identically distributed (i.i.d)  sampling centers   $\{X_j\}_{j=1}^N$ following a probability measure $\mu$ on $\Omega$ satisfying certain regularity conditions.

\begin{assumption}
\label{assumption:A} This assumption consists of two parts:
 \begin{enumerate}
     \item  There exists a positive constant $C_1$
depending only on the measure $\mu$ and dimension $d$, such that for all $x\in \Omega$ and $ \alpha \ge 0$, we have 
\begin{equation*}\label{meaure condition}
 \mu\left(\Omega \cap (B(x, \alpha)\right) \ge C_1 m\left(B(x, \alpha)\right). 
\end{equation*}
Here  
$B(x, \alpha)$ denotes  the ball centered at $x$ with radius $\alpha,$ and  $m(A)$ is the Lebesgue measure of a measurable subset $A \subset \R^d.$
\item The measure $\mu$ is extendable to be a finite positive measure $\mu^*$
on $\R^d$ such that $\widehat{\mu^*} \in L_1(\R^d).$ 
\end{enumerate}
\end{assumption}
A measure satisfying the condition of Part (1) is referred to as a “doubling" measure; see \cite[p. 8]{stein1993harmonic}. Let $\mu$ be an absolutely continuous measure on $\Omega$ (with respect to the Lebesgue measure). By Lebesgue-Radon-Nikodym theorem, there is a $g \in L_1(\Omega), \; g(x) \ge 0, \forall x \in \Omega$, such that 
\[
\mu(A) = \int_A g(x) dm(x),
\]
where $A$ is an arbitrary $\mu$-measurable subset of $\Omega$. The function $g$ is uniquely determined up to a set of Lebesgue measure zero and called the Radon-Nikodym derivative of $\mu$ with respect to $m$. In probability theory and its related areas, the function $g$ is often called the density function. If $g$ is sufficiently smooth, then
we can extend it judiciously so that the extension $g^*$ throughout 
satisfies the requirement $\widehat{g^*}\in L_1(\R^d). $
For example, if $\mu$ is the truncated Gaussian distribution, then we just extend the density function naturally so that $\mu^*$ is a Gaussian measure on $\R^d$. In addition, based on reference \cite[Prop. 2.5 (xii)]{ken1999levy},  Part (2) implies that  $\mu^*$ is absolutely continuous with respect to the Lebesgue measure and  has a bounded continuous density $h(x)$ defined via
$$
h(x)=(2\pi)^{-d}\int_{\mathbb{R}^d} e^{i\xi x}\widehat{\mu^*}(\xi) d\xi.
$$ 

We employ kernels stemming from functions defined on $\R^d$ via translation. More precisely, a function $\psi$ on $\R^d$: $x \mapsto \psi(x)$, $\psi \in C(\mathbb{R}^d
)$, induces the kernel on $\R^d \times \R^d$: $(x,y) \mapsto \psi(x-y),\; (x,y) \in \R^d \times \R^d.$
Moreover, we assume that   $\psi$ is a nonnegative approximate identity in the sense that $\int_{\mathbb{R}^d} \psi\left(x\right) dx =1$, and satisfies the conditions:  $ \int_{\mathbb{R}^d} |x|\psi\left(x\right) dx < \infty$ and $ \psi(x) \geq C_{\psi}, \ \forall \ |x| \leq 1$, where $C_\psi$ is a positive constant depending only on $\psi$. For such a kernel, we have the following two lemmas.
\begin{lemma}
\label{key_condition}
Let  $\psi$ and $\mu$ satisfy above conditions. Then the inequalities
\begin{equation*}\label{eq:kerenel_upper_bound}
\int_{\Omega}\left|\frac{x-t}{h}\right|^j \psi_h\left(x - t\right) d\mu(t) \lesssim 1,\ j=0,1
\end{equation*} hold true for each fixed $x\in \Omega$ and $0<h<1$.
\end{lemma}
\begin{proof}
Let $\mu^*$ be the continuous extension of $\mu$ to the whole space $\R^d$.  Then, based on  Part (2) of Assumption \ref{assumption:A},  we have
$$
\int_{\Omega}\left|\frac{x-t}{h}\right|^j h^{-d} \psi\left(\frac{x - t}{h}\right) d\mu(t)\le \int_{\mathbb{R}^d}\left|\frac{x-t}{h}\right|^j  h^{-d} \psi\left(\frac{x - t}{h}\right) d\mu^{*}(t):=u_{j,h} (x).$$
This together with the fact that $\psi \ge 0$ yields
\begin{equation*}
    \begin{split}
\int_{\mathbb{R}^d} |u_{j,h}(x)|dx
&\leq \int_{\mathbb{R}^d}\int_{\mathbb{R}^d} h^{-d}\left|\frac{x-t}{h}\right|^j\psi\left( \frac{x-t}{h}\right)dxd\mu^*(t)\\
&=  \int_{\mathbb{R}^d} \int_{\mathbb{R}^d} |z|^j\psi(z)dzd\mu^*(t)\\
&= \mu^{*}(\mathbb{R}^d)\int_{\mathbb{R}^d} |z|^j\psi(z)dz<\infty 
    \end{split}
\end{equation*}
by a change of variable $z=(x-t)/h$. This implies  that $u_{j,h}\in L_1(\mathbb{R}^d)$ and thus  has a  Fourier transform 
\begin{equation*}
    \begin{split}
\widehat{u_{j,h}}(\xi):=\int_{\mathbb{R}^d}u_{j,h}(x) e^{-ix\xi}dx
=\int_{\mathbb{R}^d}\left(\int_{\mathbb{R}^d} h^{-d}\left|\frac{x-t}{h}\right|^j\psi\left( \frac{x-t}{h}\right)d\mu^*(t)\right) e^{-ix\xi}dx.
    \end{split}
\end{equation*}
Further, applying the change of variable to the above  double integral, we have 
\begin{equation*}
\label{final_ Fourier_transform}
    \begin{split}
      \widehat{u_{j,h}}(\xi)=\int_{\mathbb{R}^d} \left(\int_{\mathbb{R}^d} |z|^j\psi(z)e^{-i(hz+t)\xi}dz\right) d\mu^*(t)
     & = \left(\int_{\mathbb{R}^d} e^{-it\xi}d\mu^*(t)\right) \left( \int_{\mathbb{R}^d}  |z|^j\psi(z)e^{-ihz\xi}dz\right)\\
      &= \widehat{\mu^*}(\xi)\int_{\mathbb{R}^d}  |z|^j\psi(z)e^{-ihz\xi}dz:=v_{j,h}(\xi).
    \end{split}
\end{equation*}
This implies that  $u_{j,h}(x)$ can be viewed as  the inverse Fourier transform of  $v_{j,h}(\xi)$.  Thus we have
\begin{equation*}
    \begin{split}
\left|u_{j,h}(x)\right| &\le (2\pi)^{-d}\int_{\mathbb{R}^d}|v_{j,h}(\xi)|d\xi\\
&\leq (2\pi)^{-d}\int_{\mathbb{R}^d} |\widehat{\mu^*}(\xi)| \left| \int_{\mathbb{R}^d}  |z|^j\psi(z)e^{-ihz\xi}dz\right| d\xi\\
&\le (2\pi)^{-d}\left(\int_{\mathbb{R}^d} |\widehat{\mu^*}(\xi)|d\xi\right)  \left(\int_{\mathbb{R}^d}  |z|^j\psi(z)dz \right).
    \end{split}
\end{equation*}    
This together with the conditions  $\widehat{\mu^*} \in L_1(\mathbb{R}^d)$ and  $\int_{\mathbb{R}^d}  |z|^j\psi(z)dz < \infty$ completes the proof. 
\end{proof}

\begin{lemma}
\label{mbound}
Let $\psi$ and $\mu$ satisfy above conditions. Then the following inequality  holds true 
$$\left\Vert\left(\int_{\Omega}\psi_h(\cdot-t)d\mu(t)\right)^{-1}\right\Vert_{\infty} \lesssim 1 .$$

\end{lemma}
\begin{proof}
For each fixed $x\in \Omega$, we define a corresponding neighborhood centered at $x$ via
$Z_{x, h,j}:=\Omega \cap \{t: (j-1)h\leq |x-t| <jh\}$.
Then there exists a positive integer $J_1$ such that   $\Omega= \bigcup_{j=1}^{J_1} Z_{x,h,j}$, which leads to the decomposition 
\begin{equation*}
    \begin{split}
\int_{\Omega} h^{-d}
 \psi\left(\frac{x-t}{h}\right) d\mu(t)=\sum_{j=1}^{J_1}\int_{Z_{x,h,j}}h^{-d}\psi\left(\frac{x-t}{h}\right)d\mu(t).  
    \end{split}
\end{equation*}
In addition, due to the fact that   $ \psi\geq 0$ , we   have
\begin{equation*}
    \begin{split}
\int_{\Omega} h^{-d}
 \psi\left(\frac{x-t}{h}\right) d\mu(t)\ge h^{-d}\int_{Z_{x,h,1}}\psi\left(\frac{x-t}{h}\right)d\mu(t). 
    \end{split}
\end{equation*}
This together with $\psi(x) \geq C_{\psi}\ ,\forall \ |x| \leq 1$,  and Part (1) of Assumption \ref{assumption:A} leads to
$$
\int_{\Omega} h^{-d}
 \psi\left(\frac{x-t}{h}\right) d\mu(t) \ge C_{\psi}h^{-d}\int_{Z_{x,h,1}}d\mu(t)=C_{\psi}h^{-d} \mu(Z_{x,h,1})\geq  \frac{\pi^{d/2}}{\Gamma(d/2+1)}C_{\psi} C_{1}.  
$$
Thus, the lemma holds. 
\end{proof}
 
We now give a brief introduction of     the bounded difference inequality  \cite{mcdiarmid1989method} and Bernstein inequality 
 \cite{cucker2002mathematical}  that will be used in deriving $L^p$-probabilistic convergence rates of our stochastic quasi-interpolation.  Let  $(X_1,\cdots, X_N)\in \Omega^N$ be a random vector with  independent entries. Let $\Psi: \Omega^N \rightarrow \R$  be a function satisfying the bounded difference condition\footnote{We say a function $\Psi: \Omega^N \rightarrow \R$  satisfies  the  bounded difference condition if its values can change by at most $a_i$ when its $i$th entry is modified. Namely, the inequality
\begin{equation*}\label{boundeddifferenc}    \sup_{x_1,...,x_N,x_i'\in \Omega} |\Psi(x_1,...,x_i,...,x_N)-\Psi(x_1,...,x_i',...,x_N)|\leq a_i, \ \forall\ 1\leq i\leq N,
\end{equation*} holds true for some prescribed nonnegative constants $a_1,...,a_N$.}, then we have the  bounded difference inequality    \cite{mcdiarmid1989method}
\begin{equation}\label{difference inequality}
\mathbb{P}\left(\left|\Psi(X_1,...,X_N)-\mathbb{E}[\Psi(X_1,...,X_N)] \right| \geq \epsilon \right) \leq 2 \exp{\left\{-\frac{2\epsilon^2}{\sum_{i=1}^N a_i^2}\right\}}.
\end{equation}
   Moreover, if $\{X_j\}$ additionally satisfy  $\mathbb{P}\{|X_j|\le M\}=1, \ j=1,\cdots, N$, for some positive constant $M$, then for any given  $\epsilon>0$, we have the   well-known Bernstein inequality \cite{cucker2002mathematical}  
\begin{equation}\label{Bernsteininequality}
\mathbb{P}\{|\bar{X}_N-\mathbb{E}[\bar{X}_N]|\ge \epsilon\} \le 2\exp\left\{ \frac{-N\epsilon^2}{2\sigma^2+2M\epsilon/3} \right\}. 
\end{equation} 
Here  $\bar{X}_N=\frac{1}{N}\sum_{j=1}^NX_j$, $\sigma^2=\frac{1}{N}\sum_{j=1}^N\text{Var}(X_j)$ are corresponding mean and variance of the random vector  $(X_1,\cdots, X_N)$.

\section{The main results}
In this section, we  shall provide a general framework for constructing quasi-interpolation from the functional-optimization point of view.  Following references \cite{gao2020optimality}, \cite{gao2020multivariate},  \cite{wu2005generalized},  we   recast quasi-interpolation as a two-step procedure: first convolution and then discretization.  More precisely,  we first approximate a target function via a convolution sequence with a scaled  kernel  and then discretize the convolution sequence using certain quadrature rules.  Moreover, we show that both the convolution sequence and its discretization are  solutions of some appropriately defined functional-optimization problems.  We then derive corresponding approximation error using the well-known bias-variance decomposition   \cite{gao2020optimality}  or approximation error and sampling error decomposition in machine learning \cite{smaleZhou2007-ConstrApprox-learning} . We note that such a two-step viewpoint is only for  better understanding of quasi-interpolation, the final quasi-interpolant does not  involve any convolution procedure.
\subsection{Approximation via convolution }
Let $\psi$ be a kernel satisfying conditions in Section $2$. Then it is easy to verify that the kernel
\begin{equation*}
\label{psi_h^*}
\psi_h^*(x-y):=\frac{\psi_h(x-y) }{\int_{\Omega}\psi_h(x-t)d\mu(t)}
\end{equation*}
satisfies the identity $$\int_\Omega \psi_h^*(x-y) d\mu(y)\equiv 1, \forall  x \in \Omega.$$ Moreover,  with $\psi_h^*$ being at hand,  we  construct  a convolution sequence  $f_{\rm min}$ as
\begin{equation}
\label{eq:convolution_operators}
f_{\rm min}(x)=(f*\psi_h^*)(x)=\int_{\Omega}f(y)\psi_h^*(x-y)d\mu(y). \end{equation}
Further, we shall show that $f_{\rm min}$  is  the unique solution of the functional-optimization problem defined in Formula \eqref{variationalprobleminte}.
\begin{lemma}\label{optimality}
Let $f_{\rm min}$ be defined in  Formula \eqref{eq:convolution_operators}.  Then  we have
    \begin{equation*}
      f_{\rm min}(x)=\arg\min_{f_0\in C(\Omega)} \int_{\Omega}\int_{\Omega}\left[f(t)-f_0(x)\right]^2\psi_h(x-t)d\mu(t)dx.
    \end{equation*}
\end{lemma}
\begin{proof}  
For any  test function $ \eta(x) $, taking the variational derivative  with respect to $ f_0(x) $ , we  have
$$
\begin{aligned}
   &  \lim_{\epsilon \to 0} \frac{\int_{\Omega}\int_\Omega (f(t) - f_0(x) - \epsilon \eta(x))^2 \psi_h(x-t)\,   d\mu(t)dx - \int_{\Omega}\int_{\Omega} (f(t) - f_0(x))^2 \psi_h(x- t) \,  d\mu(t)dx}{\epsilon}\\
  &   =\lim_{\epsilon \to 0} \frac{ \int_{\Omega}\int_{\Omega} (-2\epsilon\eta(x)(f(t) - f_{0}(x))+\epsilon^2\eta^2(x)) \psi_h(x-t) \,  d\mu(t)dx}{\epsilon}\\
  &= -2 \int_{\Omega} \int_{\Omega}\eta(x)(f(t) - f_0(x)) \psi_h(x-t) \,  d\mu(t)dx.
\end{aligned}
$$
Therefore, by setting the above variational derivative to be zero,  we can get
$$
\int_{\Omega} \int_{\Omega}\eta(x)(f(t) - f_0(x)) \psi_h(x-t) \,  d\mu(t)dx = 0.
$$ Since $ \eta(x) $ is an arbitrary test function,  based on \cite{gelfand2000calculus},  we can derive
$$
\int_\Omega (f(t) - f_0(x)) \psi_h(x-t)d\mu(t) = 0 \quad \forall x \in \Omega.
$$
This in turn yields the desired result.
\end{proof}

In the following theorem, we bound  $\|f- f_{\rm min}\|_p$ by the modulus of continuity  of $f$.
\begin{theorem}\label{th:convolution_error}
    Let $f_{\rm min}$ be defined as above with $\psi$ and $\mu$ satisfying  the conditions in Section $2$. Then the following inequality 
$$
 \|f- f_{\rm min}\|_p\leq C \omega_f(h)
$$
holds true for any $1\le p\le \infty$.
\end{theorem}
\begin{proof}
Based on the Sobolev embedding inequality, for any $ 1\le p\le \infty$, we have
$$ \|f- f_{\rm min}\|_p\le (m(\Omega))^{\frac{1}{p}} \| f-f_{\rm min}\| _{\infty
}.$$
Thus it suffices to consider the  case $p=\infty$. Fixing an $x \in \Omega$, we have 
\begin{equation*}
    \begin{split}
        |f(x)-f_{\rm min}(x)| &=\left | \int_{\Omega}(f(x)-f(y))\frac{\psi_h(x-y) }{\int_{\Omega}\psi_h(x-t)d\mu(t)}d\mu(y)\right|\\
        &\leq  \int_{\Omega}|f(x)-f(y)|\frac{\psi_h(x-y) }{\int_{\Omega}\psi_h(x-t)d\mu(t)}d\mu(y).
    \end{split}   
\end{equation*}
Moreover,  using the sub-additivity of the modulus of continuity, we have
$$
|f(x)-f(y)| \leq \omega_f(|x-y|)\le (1+h^{-1}|x-y|)\omega_f(h)
$$ for any $y\in \Omega$. This in turn leads to 
\begin{equation*}
    \begin{split}
      |f(x)-f_{\rm min}(x)|
        &\leq \int_{\Omega}(1+h^{-1}|x-y|)\omega_f(h)\frac{\psi_h(x-y) }{\int_{\Omega}\psi_h(x-t)d\mu(t)}d\mu(y)\\
         &=\omega_f(h)+ \int_{\Omega}\left|\frac{x-y}{h}\right|\omega_f(h)\frac{\psi_h(x-y) }{\int_{\Omega}\psi_h(x-t)d\mu(t)}d\mu(y).
    \end{split}   
\end{equation*}
Consequently,  based on  Lemma \ref{key_condition} and Lemma  \ref{mbound}, we obtain 
\begin{equation}
\label{eq: f_bisa_part}
 |f(x)-f_{\rm min}(x)| \leq C \omega_f(h).  
\end{equation}
Since the above inequality holds true for any $x \in \Omega$, we get the desired result.
\end{proof}
In what follows, we will denote $C$, unless otherwise specified, as universal constant
whose value may change from line to line.

However, the convolution sequence  $f_{\rm min}$ is only an intermediate approximation step for  theoretical analysis. Note that it is impossible to get an explicit expression of $\int_{\Omega}\psi_h(x-t)d\mu(t)$ for most kernels $\psi$. In addition, we can only have discrete function values  evaluated at   sampling centers in practical applications.  Therefore,  we have to discretize $f_{\rm min}$  to get a user-friendly approximation scheme.  In the next subsection, we shall discretize  both the numerator and denominator in $f_{\rm min}$  simultaneously  based on  random sampling data to  get a stochastic quasi-interpolation scheme in a rational form.
\subsection{Stochastic quasi-interpolation}

Let $ X $ be a random variable distributed over $ \Omega $ according to the probability  distribution $ \mu $. Denote by $ \{X_j\}_{j=1}^N $ a collection of $ N $ independent copies of $ X $,  let $ \{f(X_j)\}_{j=1}^N $ and  $ \{\psi_h(x-X_j)\}_{j=1}^N $ be corresponding discrete function values at  $ \{X_j\}_{j=1}^N $, respectively.  Then,  applying the Monte Carlo discretization technique to the convolution sequence  $f_{\rm min}$, we get an ansatz in the scaled Shepard's form   \cite{senger2018optimal}
\begin{equation*}
Q_{h}^rf(x):=\sum^N_{j=1}f(X_j)\frac{\psi_h(x- X_j)}{\sum_{\ell=1}^N\psi_h(x-X_{\ell})},  \ \forall x\in \Omega.
\end{equation*}
 Our task now is to study  approximation errors of $Q_h^rf$  under the framework of probabilistic numerics \cite{hennig2015probabilistic}.

Let
\begin{equation*}Q_hf(x)=\frac{1}{N}\sum_{j=1}^Nf(X_j)\psi_h(x-X_j)  
\end{equation*}
and $$r_h(x)=\frac{1}{N}\sum_{l=1}^N\psi_h(x-X_l),$$  then we have the following lemma.

\begin{lemma}
\label{le:decomposition}
   Let $Q_h^rf$, $Q_hf$ and $r_h$ be defined as above with $\psi$ and $\mu$ satisfying conditions in Section $2$. Then, for any given   $\epsilon>0$,  we can choose  $ C\omega_f(h) < \frac{\epsilon}{2}$, 
such that the probability inequality
\begin{equation*}
        \mathbb{P}\{\|Q_h^rf-f\|_{p}>\epsilon\}
        \leq \mathbb{P}\left\{\|Q_hf-\mathbb{E}[Q_hf]\|_{p}\geq \frac{\epsilon}{4C}\right\}+\mathbb{P}\left\{\|r_h-\mathbb{E}[r_h]\|_{p}\geq \frac{\epsilon}{4C}\right\}
\end{equation*}
holds true for any $1\leq p \leq \infty$.
\end{lemma}
\begin{proof}
  Using the triangle inequality, we have 
\begin{equation}
\label{eq:decomposition}
 \|Q_h^rf-f\|_{p}
  \leq \left\Vert  f -f_{\rm min}\right\Vert_{p}+  \left\Vert Q_h^rf-f_{\rm min}\right\Vert_{p}.
\end{equation}
Observe that
\begin{equation}
\label{Q_h^rf(x)}
\begin{split}
Q_h^rf(x)&=\frac{1}{r_h(x)}Q_hf(x)\left(\frac{r_h(x)}{\mathbb{E}[r_h(x)]}+1-\frac{r_h(x)}{\mathbb{E}[r_h(x)]}\right)\\
&=\frac{1}{\mathbb{E}[r_h(x)]}Q_hf(x)+\frac{1}{\mathbb{E}[r_h(x)]}(\mathbb{E}[r_h(x)]-r_h(x))Q_h^rf(x).
\end{split}
\end{equation}
We obtain
\begin{equation}
\label{eq:l_p_second_part}
    \begin{split}
\left\Vert Q_h^rf-f_{\rm min}\right\Vert_{p} &=\left\Vert \frac{1}{\mathbb{E}[r_h]}Q_hf+\frac{1}{\mathbb{E}[r_h]}(\mathbb{E}[r_h]-r_h)Q_h^rf- \frac{1}{\mathbb{E}[r_h]}\mathbb{E}[Q_hf] \right\Vert_{p}\\
& \leq \Vert  (\mathbb{E}[r_h])^{-1}\Vert_{\infty}\left\Vert Q_hf-\mathbb{E}[Q_hf] +(\mathbb{E}[r_h]-r_h)Q_h^rf\right\Vert_{p}\\
& \leq  \Vert  (\mathbb{E}[r_h])^{-1}\Vert_{\infty}(\left\Vert Q_hf-\mathbb{E}[Q_hf]\right\Vert_{p} + \|f\|_{\infty}\left\Vert r_h-\mathbb{E}[r_h]\right\Vert_{p})\\
& \leq C  (\|Q_hf-\mathbb{E}[Q_hf]\|_{p}+\|r_h-\mathbb{E}[r_h]\|_{p}),
    \end{split}
\end{equation}
based on  Lemma \ref{mbound}.
 Therefore, according to Inequalities \eqref{eq:decomposition} and  \eqref{eq:l_p_second_part}, we have 
\begin{equation*}
    \begin{split}
         \|Q_h^rf-f\|_{p}
        \leq \| f-f_{\rm min}\| _{p}+C(\|Q_hf-\mathbb{E}[Q_hf]\|_{p}+\|r_h-\mathbb{E}[r_h]\|_{p}).
    \end{split}
\end{equation*}
In addition, according to  Theorem \ref{th:convolution_error},  the assumption $C\omega_f(h)  < \frac{\epsilon}{2}$ leads to
$$
\| f-f_{\rm min}\| _{p}<\frac{\epsilon}{2}.
$$
Then we have
$$ \left\{(X_1,\cdots,X_N): \|Q_h^rf-f\|_{p}>\epsilon\right\} \subseteq \left\{(X_1,\cdots,X_N):  \|Q_hf-\mathbb{E}[Q_hf]\|_{p}+\|r_h-\mathbb{E}[r_h]\|_{p}\geq \frac{\epsilon}{2C}\right\}.$$
This in turn leads to
\begin{equation*}
\begin{split}
        \mathbb{P}\{\|Q_h^rf-f\|_{p}>\epsilon\}&\leq  \mathbb{P}\left\{\|Q_hf-\mathbb{E}[Q_hf]\|_{p}+\|r_h-\mathbb{E}[r_h]\|_{p}\geq \frac{\epsilon}{2C}\right\}\\
        &\leq \mathbb{P}\left\{\|Q_hf-\mathbb{E}[Q_hf]\|_{p}\geq \frac{\epsilon}{4C}\right\}+\mathbb{P}\left\{\|r_h-\mathbb{E}[r_h]\|_{p}\geq \frac{\epsilon}{4C}\right\}.
\end{split}
\end{equation*}  
Thus, the lemma holds. 
\end{proof}
Observe that the $r_h(x)$ is a special case of $Q_hf(x)$ with $f$  whose  restriction over $\Omega$ is the characteristic function (i.e., $f|_{\Omega}\equiv1$). 
 Based on the above lemma, it is sufficient to derive the bounds of $\mathbb{P}\left\{\|Q_hf-\mathbb{E}[Q_hf]\|_{p}\geq \frac{\epsilon}{4C}\right\}$. Moreover, according to the Sobolev embedding inequality $\|Q_hf-\mathbb{E}[Q_hf]\|_{p}\leq (m(\Omega))^{\frac{1}{p}}\|Q_hf-\mathbb{E}[Q_hf]\|_{\infty}$,
we only need to derive the bound of  $\mathbb{P}\left\{\|Q_hf-\mathbb{E}[Q_hf]\|_{\infty}\geq \frac{\epsilon}{4C} \right\}$. 

Following  the partition technique proposed by Bourgain and Lindenstrauss \cite{bourgain1988distribution}, we first partition the domain $\Omega$ into  $J$ small subdomains $\{I_{h,j}\}_{j=1}^J$ whose interiors do not intersect, that is, $\overset{\circ}{\mathrm{I}}_{h,j} \cap \overset{\circ}{\mathrm{I}}_{h,j'}=\varnothing,\ j\neq j'$.   Here $J$ is a finite positive integer such that $\Omega=\bigcup_{j=1}^{J}\mathrm{I}_{h,j}$.  Moreover, we assume that these subdomains satisfy certain regularity conditions:  $m(\mathrm{I}_{h,j})\gtrsim  h^{d(d+2)}$ and $\max_{1\leq j\leq J}\{\text{Diam}(\mathrm{I}_{h,j})\}\lesssim  h^{d+2} $ with $\text{Diam}(\cdot)$ denoting  the  diameter of a set. 
Then  we have  
$$ \left\{(X_1,\cdots,X_N):  \max_{y\in \Omega}|Q_hf(y)-\mathbb{E}[Q_hf(y)]|\geq \frac{\epsilon}{4C}\right\} \subseteq\bigcup_{j \in J}\left\{ (X_1,\cdots,X_N): \max_{y_j\in \mathrm{I}_{h,j}}|Q_hf(y_j)-\mathbb{E}[Q_hf(y_j)]|\geq \frac{\epsilon}{4C}\right\}.$$  
This in turn leads to 
\begin{equation}
\label{eq:domain_partition}
    \begin{split}
        \mathbb{P}\left\{\max_{y\in \Omega}|Q_hf(y)-\mathbb{E}[Q_hf(y)]|\geq \frac{\epsilon}{4C}\right\}&\leq   \mathbb{P}\left\{\bigcup_{j=1}^J\left\{\max_{y_j\in I_{h,j}}|Q_hf(y_j)-\mathbb{E}[Q_hf(y_j)]|\geq \frac{\epsilon}{4C}\right\}\right\}\\
        &\leq \sum_{j=1}^J\mathbb{P}\left\{\max_{y_j\in I_{h,j}}|Q_hf(y_j)-\mathbb{E}[Q_hf(y_j)]|\geq \frac{\epsilon}{4C}\right\}.
\end{split}
\end{equation}
We now derive the decay rate of $\mathbb{P}\left\{\max_{y_j\in I_{h,j}}|Q_hf(y_j)-\mathbb{E}[Q_hf(y_j)]|\geq \frac{\epsilon}{4C} \right\}$. Based on the triangle inequality and Jensen’s inequality, we have 
    \begin{equation*}
        \begin{split}
            |Q_hf(y_j)-\mathbb{E}[Q_hf(y_j)]|
            &\leq |Q_hf(t_j)-\mathbb{E}[Q_hf(t_j)]|+|Q_hf(y_j)-Q_hf(t_j)|+\mathbb{E}|Q_hf(y_j)-Q_hf(t_j)|
        \end{split}
    \end{equation*}
     for a fixed $t_j\in \mathrm{I}_{h,j}$ and any $ y_j \in \mathrm{I}_{h,j}$. Furthermore, let $v_j $ be the direction vector from  $t_j$ to   $ y_j$ and $D_{v_j}$ the corresponding directional derivative, then we have 
 \begin{equation*}
          |Q_hf( y_j)-Q_hf(t_j)|
          \le h^{-d}\|f\|_{\infty} \|D_{v_j} \psi\|_{\infty}h^{-1}|t_j- y_j|\leq Ch.
 \end{equation*}
This in turn leads to
$$
     \mathbb{E} |Q_hf( y_j)-Q_hf(t_j)|
          \leq Ch,$$
and
 $$
  |Q_hf( y_j)-\mathbb{E}[Q_hf( y_j)]|\leq |Q_hf(t_j)-\mathbb{E}[Q_hf(t_j)]|+2Ch.
 $$
 Moreover, since the above inequality holds for any $ y_j \in \mathrm{I}_{h,j}$, we obtain 
 $$
  \max_{ y_j\in \mathrm{I}_{h,j}}|Q_hf( y_j)-\mathbb{E}[Q_hf( y_j)]|\leq  |Q_hf(t_j)-\mathbb{E}[Q_hf(t_j)]|+2Ch.
 $$
If we further assume that  $2Ch< \frac{\epsilon}{8}$, then we have
\begin{equation*}
    \mathbb{P}\left\{\max_{ y_j\in \mathrm{I}_{h,j}}|Q_hf( y_j)-\mathbb{E}[Q_hf( y_j)]|\geq \frac{\epsilon}{4C}\right\}\le \mathbb{P}\left\{|Q_hf(t_j)-\mathbb{E}[Q_hf(t_j)]|\geq \frac{\epsilon}{8C}\right\}.
\end{equation*}
This together with Inequality \eqref{eq:domain_partition} yields 
\begin{equation}
\label{eq:pointwise}
    \begin{split}
    \mathbb{P}\left\{\|Q_hf-\mathbb{E}[Q_hf]\|_{\infty}\geq \frac{\epsilon}{4C}\right\}&\leq\sum_{j=1}^J \mathbb{P}\left\{|Q_hf(t_j)-\mathbb{E}[Q_hf(t_j)]|\geq \frac{\epsilon}{8C}\right\}\\
    &\leq C h^{-d(d+2)}\mathbb{P}\left\{|Q_hf(t_{j^*})-\mathbb{E}[Q_hf(t_{j^*})]|\geq \frac{\epsilon}{8C}\right\}
    \end{split}
\end{equation}
with
$$\mathbb{P}\left\{|Q_hf(t_{j^*})-\mathbb{E}[Q_hf(t_{j^*})]|\geq \frac{\epsilon}{8C}\right\}=\max_{1\leq j\leq J}\left\{\mathbb{P}\left\{|Q_hf(t_j)-\mathbb{E}[Q_hf(t_j)]|\geq \frac{\epsilon}{8C}\right\}\right\}.$$

 Our task now is to derive the bound of  $\mathbb{P}\left\{|Q_hf(t_{j^*})-\mathbb{E}[Q_hf(t_{j^*})]|\geq \frac{\epsilon}{8C}\right\}$. Let $Z_j(t_{j^*})=f(X_j)\psi_h(t_{j^*}-X_j)$, then it is easy to get $|Z_j(t_{j^*})|\le \Vert f\Vert_{\infty}\Vert \psi \Vert_{\infty,\mathbb{R}^d}h^{-d}$ and 
 $$\frac{1}{N}\sum_{j=1}^N\text{Var}(Z_j(t_{j^*}))\le \Vert f\Vert_{\infty}^2h^{-d}\mathbb{E}[h^{-d}\psi^2(h^{-1}(t_{j^*}-X_1))]\leq Ch^{-d}$$
based on  Lemma \ref{key_condition}.
Thus, according to  Bernstein inequality \eqref{Bernsteininequality}, we have 
$$
\mathbb{P}\left\{|Q_hf(t_{j^*})-\mathbb{E}[Q_hf(t_{j^*})]|\geq  \frac{\epsilon}{8C}\right\}\le 2\exp\left\{-\frac{CNh^d\epsilon^2}{1+\epsilon}\right\}.
$$
These together with Inequlity \eqref{eq:pointwise}   lead to  
 \begin{equation*}
    \begin{split}       
        \mathbb{P}\left\{\|Q_hf-\mathbb{E}[Q_hf]\|_{\infty}\geq  \frac{\epsilon}{4C}\right\}
        &\leq C(h^{d+2})^{-d}\exp\left\{-\frac{CNh^d\epsilon^2}{1+\epsilon}\right\} \\
        &=  C\exp\{-d(d+2)\mathrm{log} \ h\}\exp\left\{-\frac{CNh^d\epsilon^2}{1+\epsilon}\right\}\\
        &\leq C  \exp\left\{-Nh^d\epsilon^2\left(\frac{C}{1+\epsilon}-\frac{d(d+2)|\mathrm{log} \ h|}{Nh^d\epsilon^2}\right)\right\}.
 \end{split}
\end{equation*}
Furthermore, if we assume $CN^{-\frac{1}{2}}h^{-\frac{d}{2}}|\mathrm{\log} \ h|^{\frac{1}{2}}<\epsilon$, then we have  $\frac{d(d+2)|\log h|}{Nh^d\epsilon^2}\le \frac{C}{2(1+\epsilon)}$ by some simple derivations. Finally,  we have
\begin{equation*}
 \mathbb{P}\left\{\|Q_hf-\mathbb{E}[Q_hf]\|_{\infty}\geq  \frac{\epsilon}{4C}\right\}\leq C \exp\left\{-\frac{CNh^d\epsilon^2}{2(1+\epsilon)}\right\}.  
\end{equation*}

We conclude the above analysis into theorem.
\begin{theorem}\label{convergenceanalysi_infty}
 Let $Q_h^rf$ be defined as above with $\psi$ and $\mu$ satisfying conditions in Section $2$. Then, for any given $\epsilon>0$, we can choose  
 \begin{equation*}
  C\omega_f(h)  < \frac{\epsilon}{2} \  \text{and} \ CN^{-\frac{1}{2}}h^{-\frac{d}{2}}|\mathrm{log} \ h|^{\frac{1}{2}}< \epsilon,
 \end{equation*}
 such that the inequality
\begin{equation*}
    \begin{split}
          \mathbb{P}\{\|Q_h^rf-f\|_{p}>\epsilon\}
    \leq C
     \exp\left\{-\frac{CNh^d\epsilon^2}{2(1+\epsilon)}\right\}
    \end{split}
\end{equation*}
 holds true for any $1\le p\le \infty$.   
\end{theorem}
Up to now, we have established a family of $L^p$-McDiarmid-type  concentration inequalities for $1\leq p\leq \infty$. However, we shall demonstrate that the above inequalities can be further improved for $1\le p\le 2$ with  the  bounded difference inequality.

Observing that  
$$
\| Q_hf-\mathbb{E}[Q_hf]\|_{p}\le \left|\|Q_hf-\mathbb{E}[Q_hf]\|_{p}-\mathbb{E}\| Q_hf-\mathbb{E}[Q_hf]\|_{p}\right|+\mathbb{E}\| Q_hf-\mathbb{E}[Q_hf]\|_{p},
$$
we have 
 \begin{equation}\label{key inequality}
    \begin{split}
      \mathbb{P}\left\{\|  Q_hf-\mathbb{E}[Q_hf]\| _{p}\geq \epsilon\right\} & \leq \mathbb{P}\left\{\left|\|Q_hf-\mathbb{E}[Q_hf]\|_{p}-\mathbb{E}\| Q_hf-\mathbb{E}[Q_hf]\|_{p}\right|+\mathbb{E}\| Q_hf-\mathbb{E}[Q_hf]\|_{p}\geq\epsilon\right\}.
\end{split}
\end{equation}
Therefore, it only needs to derive bounds  of  $\mathbb{E}\Vert Q_hf-\mathbb{E}[Q_hf]\Vert_{p}$ and $\mathbb P \left( \left |\Vert Q_hf-\mathbb{E}[Q_hf]\Vert_{p}-\mathbb{E}\Vert Q_hf-\mathbb{E}[Q_hf]\Vert_{p}\ \right|  \geq\epsilon\right)$. 

\begin{lemma}
\label{irr expectation super norm estimate}
Let  $Q_hf$ be defined as above with $\psi$ and $\mu$ satisfying  conditions in Section $2$.  Then, for any $1\le p \le  2$, we have
\begin{equation*}
    \mathbb{E} \Vert Q_hf-\mathbb{E}[Q_hf] \Vert_{p}
    \leq C 
    N^{-\frac{1}{2}}h^{-\frac{d}{2}}.  
\end{equation*} 
\end{lemma}
\begin{proof}
   We  first use Sobolev embedding inequality and then  Jensen's inequality \cite{hardy1952inequalities}  to write
   \begin{equation*}
   \mathbb{E}\Vert Q_hf-\mathbb{E}[Q_hf] \Vert_{p} \le |\Omega|^{\frac{1}{p}-\frac{1}{2}}\mathbb{E}\Vert Q_hf-\mathbb{E}[Q_hf] \Vert_{2}\le |\Omega|^{\frac{1}{p}-\frac{1}{2}}(\mathbb{E}\Vert Q_hf-\mathbb{E}[Q_hf] \Vert_{2}^2)^{\frac{1}{2}}.
   \end{equation*}
 Based on  Fubini's theorem \cite{fubini1907sugli},  we have 
\begin{equation}
\label{eq:p_embedding}
   \mathbb{E}\Vert Q_hf-\mathbb{E}[Q_hf] \Vert_{p}\le |\Omega|^{\frac{1}{p}-\frac{1}{2}}\left(\int_\Omega\mathbb{E}| Q_hf(x)- \mathbb{E}[Q_hf(x)]|^2 dx\right)^{\frac{1}{2}}.
\end{equation}
This together with the fact that $\{  X_j\}_{j=1}^N$ are independent and identically distributed random variables yields
    \begin{equation*}
            \mathbb{E}[Q_hf(x)-\mathbb{E}[Q_hf(x)]]^2=\text{Var}\left( \frac{1}{N}\sum_{j=1}^Nf(X_j)\psi_h(x-X_j)\right)
            \le\frac{\|f\|_{\infty}^2}{Nh^{2d}}\text{Var}(\psi(h^{-1}(x-X))).
    \end{equation*}
Furthermore, observing that  $\text{Var}(\psi(h^{-1}(x-X)) \le E[\psi(h^{-1}(x-X))]^2$, we have 
\begin{equation*}  
 \text{Var}(Q_hf(x))\leq \frac{\|f\|_{\infty}^2}{Nh^{d}}\mathbb{E}[h^{-d}\psi^2(h^{-1}(x-X))] \leq CN^{-1}h^{-d},
\end{equation*}
where we have used  Lemma \ref{key_condition}.
Thus the lemma holds.
\end{proof}
\begin{lemma}\label{expo1}
Let  $Q_hf$ be defined as above with $\psi$ and $\mu$ satisfying conditions in Section $2$. Then   we have the following inequality
$$
\mathbb P \left( \left |\ \Vert Q_hf-\mathbb{E}[Q_hf]\Vert_{p}-\mathbb{E}\Vert Q_hf-\mathbb{E}[Q_hf]\Vert_{p}\ \right| \geq \epsilon\right) \leq 2\exp\left\{-CNh^{2d\left(1-\frac{1}{p}\right)}\epsilon^2\right\}
$$ holds true for any given  $\epsilon>0$.
\end{lemma}
\begin{proof}
Let $g(X_1,...,X_N)=\Vert Q_hf-\mathbb{E}[Q_hf]\Vert_{p}$.   Based on Inequality \eqref{difference inequality}, we only need to  show that $g$ satisfies the bounded difference condition\footnotemark[\value{footnote}].  This follows from the triangle inequality that
\begin{equation*}
      \begin{split}
|g(x_1,...,x_i,...,x_N)-g(x_1,...,x_{i}^{'},...,x_N)|=
& \left | \| Q_hf(\cdot;x_1,...,x_i,...,x_N)-\mathbb{E}[Q_hf]\| _{p}- \| Q_hf(\cdot;x_1,...,x_i^{'},...,x_N)-\mathbb{E}[Q_hf]\| _{p}\right |\\
\leq & \| Q_hf(\cdot;x_1,...,x_i,...,x_N)- Q_hf(\cdot;x_1,...,x_i^{'},...,x_N)\| _{p}\\
= & \left \Vert \frac{1}{N}f(x_i)\psi_h(\cdot-x_i)- \frac{1}{N}f(x_i^{'})\psi_h(\cdot-x_i{'})\right \Vert_{p}\\
\leq &\frac{\|f\|_{\infty}}{Nh^d} \left(\left \Vert \psi(h^{-1}(\cdot-x_i))\right\Vert_{p,\mathbb{R}^d} +  \left\Vert\psi(h^{-1}(\cdot-x_i{'}) )\right\Vert_{p,\mathbb{R}^d}\right)\\
\leq &2\|f\|_{\infty}\| \psi\| _{p,\mathbb{R}^d} N^{-1}h^{d/p-d}.
      \end{split}
  \end{equation*}
 Since the right hand side of the above inequality does not depend on $x_1,\cdots,x_i,\cdots,x_N,x_i^{'} \in \Omega$, we use  Inequality \eqref{difference inequality} to get the desired result.
\end{proof}

Based on these  two lemmas,  we can get an improved error estimates of $Q_h^rf$ as given in the following theorem.
 \begin{theorem}\label{convergenceanalysi_1}
 Let $Q_h^rf$ be defined as above with $\psi$ and $\mu$ satisfying conditions in Section $2$. Then, for any given   $\epsilon>0$, we can choose 
 \begin{equation*}
     \begin{split}
 C\omega_f(h) < \frac{\epsilon}{2}\ 
 \mbox{and} \ 
    CN^{-\frac{1}{2}}h^{-\frac{d}{2}}<\frac{\epsilon}{8},
     \end{split}
 \end{equation*}
 such that the inequality
\begin{equation*}
    \begin{split}
          \mathbb{P}\{\|Q_h^rf-f\|_{p}>\epsilon\}
    \leq 4\exp\left\{-C Nh ^{2d\left(1-\frac{1}{p}\right)}\epsilon^2\right\}
    \end{split}
\end{equation*}
 holds true for any $1\leq p\leq 2$.
\end{theorem}
\begin{proof}
Using Lemma \ref{irr expectation super norm estimate} and  the assumption $ CN^{-\frac{1}{2}}h^{-\frac{d}{2}}<\frac{\epsilon}{8}$,  we obtain 
 \begin{equation*}
    \begin{split}
      \mathbb{P}\left\{\|  Q_hf-\mathbb{E}[Q_hf]\| _{p}\geq \frac{\epsilon}{4}\right\} 
      &\leq \mathbb{P}\left\{\left|\| \mathbb{E}[Q_hf]-Q_hf\| _{p}-\mathbb{E}\|  \mathbb{E}[Q_hf]-Q_hf\| _{p}\right|\geq \frac{\epsilon}{8}\right\}
\end{split}
\end{equation*}
from Inequality \eqref{key inequality}. Based on 
 Lemma \ref{expo1}, we have 
 \begin{equation*}
    \begin{split} 
      \mathbb{P}\left\{\|  Q_hf-\mathbb{E}[Q_hf]\| _{p}\geq \frac{\epsilon}{4}\right\} &\leq  2\exp\left\{-C Nh ^{2d\left(1-\frac{1}{p}\right)}\epsilon^2\right\}.
    \end{split}
\end{equation*}
This together with 
 Lemma \ref{le:decomposition} yields the desire result.
\end{proof}

We end up with establishing the mean $L^p$-error estimate of $Q_h^rf$ , that is, $\mathbb{E}\Vert Q_h^rf-f\Vert_{p}$. 
 \begin{theorem}
\label{mean Lp eatimate}
   Let  $Q_h^rf$ be defined as above with $\psi$ and $\mu$ satisfying conditions in Section $2$. Then we have 
    \begin{equation*}
    \mathbb{E}\left\Vert Q_h^rf-f\right\Vert_{p}
    \lesssim \left\{\begin{array}{ll}
   \max\left\{\omega_f(h), \ N^{-\frac{1}{2}}h^{-\frac{d}{2}}\right\}, & \hbox{if} \quad 1 \le p \le 2,\\
\max\left\{\omega_f(h),  \ N^{-\frac{1}{2}}h^{-\frac{d}{2}}|\mathrm{log} \ h|^{\frac{1}{2}}\right\}, & \hbox{if} \quad 2<p \le \infty.
       \end{array} \right. 
\end{equation*} 
In particular,  if   $\omega_f(h)=\mathcal{O}{(h^s)} \ (0<s\le 1)$, then by setting $h=\mathcal{O}{(N^{-\frac{1}{2s+d}})}$, we can even get an optimal error estimate
\begin{equation*}
\label{prior_MSE}
     \mathbb{E}\left\Vert Q_h^rf-f\right\Vert_{p}
    \lesssim \left\{\begin{array}{ll}
   N^{-\frac{s}{2s+d}} , & \hbox{if} \quad 1 \le p \le 2,\\
  N^{-\frac{s}{2s+d}}(\log N)^{\frac{1}{2}}  & \hbox{if} \quad 2<p \le \infty.
       \end{array} \right. 
\end{equation*}
\end{theorem}

\begin{proof}
We first consider the case $1\le p\le 2.$   Without loss of generality, we assume $\omega_f(h)\leq CN^{-\frac{1}{2}}h^{-\frac{d}{2}}\leq \epsilon$. Then, based on the identity \cite{wu2020polynomial}
$\mathbb{E}[V]=\int_0^\infty \mathbb{P}\{ V>r\}dr$ and Theorem \ref{convergenceanalysi_1}, we have 
    \begin{equation*}
    \begin{split}
    \mathbb{E}\left\Vert Q_h^rf-f\right\Vert_{p}
    &\leq  \int_{0}^{C\ N^{-\frac{1}{2}}h^{-\frac{d}{2}}}\mathbb{P}\left\{ \left\Vert Q_h^rf-f\right\Vert_{p}>r\right\}dr+\int_{0}^{\infty}\mathbb{P}\left\{ \left\Vert Q_h^rf-f\right\Vert_{p}>r\right\}dr\\
    &\lesssim N^{-\frac{1}{2}}h^{-\frac{d}{2}}+\int_{0}^{\infty} \exp\left\{-CNh^{2d\left(1-\frac{1}{p}\right)}r^2\right\}dr\\
    &\lesssim  N^{-\frac{1}{2}}h^{-\frac{d}{2}}+N^{-\frac{1}{2}}h^{-d\left(1-\tfrac{1}{p}\right)}\\
    &\lesssim N^{-\frac{1}{2}}h^{-\frac{d}{2}}.
  \end{split}  
\end{equation*}
Next, we consider the case $2<p\le\infty$.  
Similarly, by  assuming $\omega_f(h) \leq C N^{-\frac{1}{2}}h^{-\frac{d}{2}}|\mathrm{log} \ h|^{\frac{1}{2}}:=\eta(N,h)\leq \epsilon$, we have 
 \begin{equation*}
    \begin{split}
    \mathbb{E}\left\Vert Q_h^rf-f\right\Vert_{p}
    &\leq  \int_{0}^{\eta(N,h)}\mathbb{P}\left\{ \left\Vert Q_h^rf-f\right\Vert_{p}>r\right\}dr+\int_{\eta(N,h)}^{\infty}\mathbb{P}\left\{ \left\Vert Q_h^rf-f\right\Vert_{p}>r\right\}dr\\
    &\leq \eta(N,h)+\int_{\eta(N,h)}^{\infty}   \exp\left\{-\frac{CNh^dr}{2(1/r+1)}\right\} dr\\
    &\le N^{-\frac{1}{2}}h^{-\frac{d}{2}}|\mathrm{log} \ h|^{\frac{1}{2}} + \int_{0}^{\infty} \exp\left\{-CN^{\frac{1}{2}}h^{\frac{d}{2}}|\mathrm{log} \ h|^{\frac{1}{2}}r\right\}dr \\
 &\lesssim  N^{-\frac{1}{2}}h^{-\frac{d}{2}}|\mathrm{log} \ h|^{\frac{1}{2}}+N^{-\frac{1}{2}}h^{-\frac{d}{2}}|\mathrm{log} \ h|^{-\frac{1}{2}}\\
 &\lesssim N^{-\frac{1}{2}}h^{-\frac{d}{2}}|\mathrm{log} \ h|^{\frac{1}{2}}.
  \end{split}  
\end{equation*}
Note that we have used  the inequality $\frac{1}{r}+1\le \eta(N,h)$ from the second line to the third line of the above inequality.
This completes the proof.
\end{proof}

We have constructed a stochastic quasi-interpolation scheme $Q_h^rf$ and derived its optimality and regularization properties as well as some statistical properties under the framework of probabilistic numerics. In the next section, we shall provide some numerical simulations to validate above theoretical analysis.

\section{Numerical simulations}
We consider the following three target functions:
\begin{equation*}
    f_{test}=\left\{
        \begin{array}{ll}
         |x| &  x \in [-1,1],\\
         sin(2\pi x_1)cos(2\pi x_2)sin(2\pi x_3) & x=(x_1,x_2,x_3) \in[0,1]^3,\\
         \prod_{j=1}^{11}\left(sin\left[\frac{\pi}{2}\left(x_j+\frac{j}{11}\right)\right]\right)^{\frac{5}{j}} \ & x=(x_1,\cdots,x_{11}) \in [0,1]^{11}.
        \end{array}
        \right.
\end{equation*}
As examples, we choose 
Gaussian kernel
\begin{equation*}
    \psi_{\sigma}(x)=\frac{1}{(2\pi\sigma^2)^{d/2}}e^{-\frac{\Vert x \Vert^2}{2\sigma^2}},
\end{equation*} and compactly supported kernel \cite{wu1995compactly} 
$$
\psi_{\beta}(x)=\left(\text{max}(1 - \Vert x\Vert, 0)\right)^{\beta}.
$$
Here  $ \sigma$, $\beta$ are two positive constants, while $\| \cdot\| $ represents the Euclidean norm.
We  set $\mu$ to be a multivariate truncated normal distribution  defined on  $\Omega$.  One can choose other distributions, such as Beta distribution, uniform distribution, or more generally, any distribution with a bounded density function.  We use quasi-interpolation $Q_h^rf$ with  these two kernels ( $\psi_{\sigma}$,  $\psi_{\beta}$, under  $\sigma=1,\ \beta =3$) to approximate each target function  in dimensions $d = 1, 3, 11$, respectively.

We adopt the following technique to derive a posteriori convergence orders.  We first choose $100$ random test points $\{t_l\}_{l=1}^{100}$ over $\Omega$ and simulate $1000$ times for each $N$ sampling data.    Then we respectively compute 
empirical mean $L^1$-norm approximation error:
\begin{equation*}
\mathbb{EMAE}_{L^1}(Q_h^rf)= \frac{1}{1000}\sum_{k=1}^{1000}\frac{1}{100}\sum_{l=1}^{100}|Q_{h,k}^{r}f(t_l)-f(t_l)|,
\end{equation*} 
and empirical mean $L^\infty$-norm approximation error:
\begin{equation*}
\mathbb{EMAE}_{L^\infty}(Q_h^rf)= \frac{1}{1000}\sum_{k=1}^{1000}\text{max}_{1\le l\le 100}|Q_{h,k}^{r}f(t_l)-f(t_l)|.
\end{equation*} 
  We go further with assuming 
$$
\mathbb{EMAE}_{L^1}(Q_h^rf)=K_1N^{-\delta_1}\ \text{and} \ \ \mathbb{EMAE}_{L^\infty}(Q_h^rf)= K_2N^{-\delta_2},
$$
for some  positive constants $K_1$, $K_2$, $\delta_1$ and $\delta_2$.   This together with  the log transform technique leads to the following two log-linear models (with respect to $N$):
$$
\log(\mathbb{EMAE}_{L^1}(Q_h^rf)) = \log K_1-\delta_1 \log N\ \ \text{and} \ \ 
\log(\mathbb{EMAE}_{L^\infty}(Q_h^rf)) =\log K_2-\delta_2 \log N.
$$
To get  estimated  values (posteriori convergence orders) $\bar{\delta_1}$ and $\bar{\delta_2}$ of coefficients $\delta_1$ and $\delta_2$,   we employ the  least square technique based on approximation errors evaluated at $N=2^j$, $j=6,7,8,...,11$.  
Numerical results are provided in Table \ref{table:numerical_results_posteriori_convergence_rates}.  We choose  $h=C_{\sigma}(1/N)^{1/(2+d)}$ and $h=C_{\beta}(1/N)^{1/(2+d)}$  for $\mathbb{EMAE}_{L^1}$ and $\mathbb{EMAE}_{L^\infty}$.  
Corresponding constants $\{C_{\sigma},C_{\beta}\}$ are provided in Table
\ref{table:Parameter_Choices_posteriori_convergence_rates} in the appendix by trials and errors. 
From Table \ref{table:numerical_results_posteriori_convergence_rates}, we can find that a posteriori convergence orders (boldface) are all higher than  a priori counterparts.
 \begin{table}[ht]
\caption{A priori and  posteriori convergence orders of our quasi-interpolation  with different kernels under different dimensions.}
\centering
\begin{tabular}{ccccccc} 
\toprule
 & \multicolumn{3}{c}{convergence orders of $\mathbb{EMAE}_{L^1}$} & \multicolumn{3}{c}{convergence orders of $\mathbb{EMAE}_{L^\infty}$}  \\
\cmidrule(lr){2-4} \cmidrule(lr){5-7}  
kernel/\quad d & 1 & 3 & 11 & 1 & 3 & 11 \\
\specialrule{1.5pt}{0pt}{0pt}
$\psi_{\sigma}$ & $(0.33, \textbf{0.65})$ & $(0.20, \textbf{0.32})$  & $(0.07, \textbf{0.16})$     & $(0.33, \textbf{0.69})$ & $(0.20, \textbf{0.37})$  & $(0.07, \textbf{0.16})$   \\
$\psi_{\beta}$  & $(0.33, \textbf{0.63})$ & $(0.20, \textbf{0.24})$  & $(0.07, \textbf{0.11})$  & $(0.33, \textbf{0.55})$ & $(0.20, \textbf{0.31})$  & $(0.07, \textbf{0.22})$ \\
\bottomrule
\end{tabular}
\label{table:numerical_results_posteriori_convergence_rates}
\end{table}

\begin{figure}[ht]
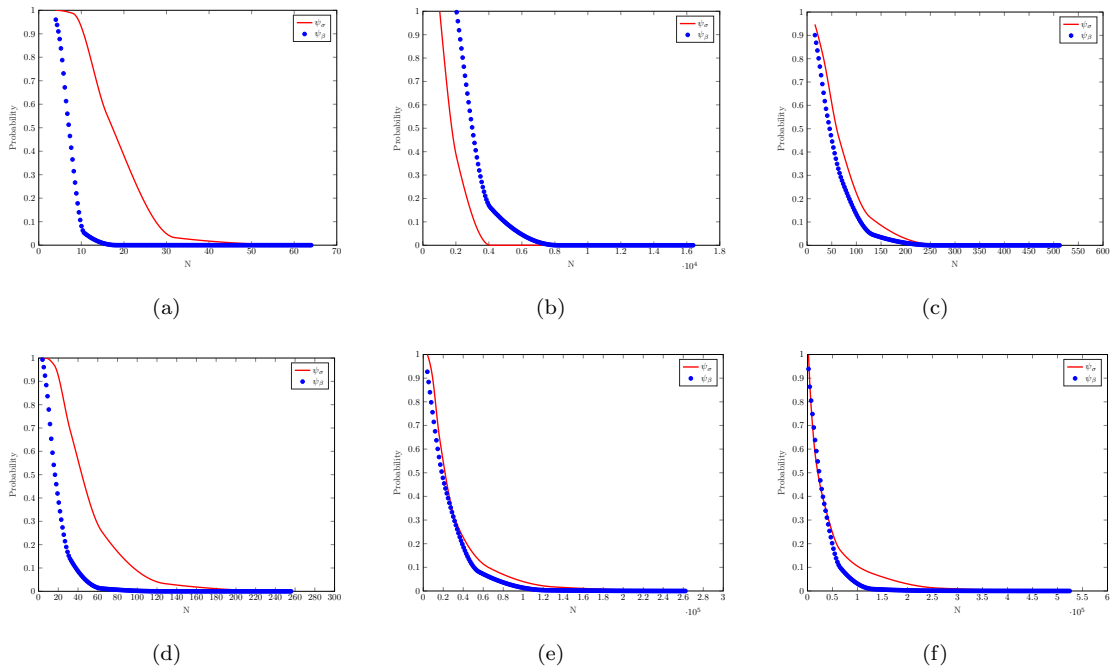

\centering
\subfigure[ ]{
\begin{minipage}[t]{0.3\linewidth}
\centering
\resizebox{0.9\linewidth}{!}{\input{1_dim_l_1}}
\end{minipage}%
}%
\subfigure[ ]{
\begin{minipage}[t]{0.3\linewidth}
\centering
\resizebox{0.9\linewidth}{!}{\input{3_dim_l_1}}
\end{minipage}%
}%
\subfigure[ ]{
\begin{minipage}[t]{0.3\linewidth}
\centering
\resizebox{0.9\linewidth}{!}{\input{11_dim_l_1}}
\end{minipage}%
}\\
\subfigure[ ]{
\begin{minipage}[t]{0.3\linewidth}
\centering
\resizebox{0.9\linewidth}{!}{\input{1_dim_l_infty}}
\end{minipage}%
}%
\subfigure[  ]{
\begin{minipage}[t]{0.3\linewidth}
\centering
\resizebox{0.9\linewidth}{!}{\input{3_dim_l_infty}}
\end{minipage}%
}%
\subfigure[  ]{
\begin{minipage}[t]{0.3\linewidth}
\centering
\resizebox{0.9\linewidth}{!}{\input{11_dim_l_infty}}
\end{minipage}%
}%
\caption{Visualization of exponential decay rates in the sense of empirical $L^1$-norm and  $L^{\infty}$-norm probabilistic convergence with respect to $N$ under different dimensions. } 
\label{Figure:Probability_simulations}
\end{figure}

We end this section with demonstrating exponential decay rates  in the sense of empirical  $L^{1}$-norm probabilistic convergence and empirical $L^\infty$-norm probabilistic convergence under threshold parameters $\epsilon=0.05$ and $\epsilon=0.1$.  
We select 100 random test points $\{t_l\}_{l=1}^{100}$ over $\Omega$. For each   $ N = 2^j, j=2,3,4,\cdots,19 $,  we respectively compute empirical probabilities of approximation errors
under 1000  simulations.  We set $ h = C_{\sigma}(1/N)^{1/(2+d)} $ and $ h = C_{\beta}(1/N)^{1/(2+d)} $   for empirical $ L^1 $-norm approximation errors and empirical $ L^\infty$-norm approximation errors. Here, constants $\{C_{\sigma},C_{\beta}\}$ are provided in Table
\ref{table:Parameter_Choices_Probability simulations} in the appendix by trials and errors.
Numerical results are shown in Figure \ref{Figure:Probability_simulations}.  Subfigures (a)-(c) demonstrate empirical probabilities in terms of empirical $L^1$-norm approximation errors  under  $d=1,3, 11$, respectively,  while   Subfigures (d)-(f)  demonstrate corresponding   $L^{\infty}$-norm  counterparts.  From these subfigures, we can see that empirical probabilities of corresponding approximation errors decays exponentially  with respect to $N$.

\bibliographystyle{amsplain} 
\bibliography{References_copy} 

\section*{Appendix}
\begin{table}[htbp]
\caption{ Choices of $\{C_\sigma,C_\beta\}$}
\centering
\begin{tabular}{ccccc} 
\toprule
 & \multicolumn{2}{c}{The $\mathbb{EMAE}_{L^1}$-case} & \multicolumn{2}{c}{The $\mathbb{EMAE}_{L^\infty}$-case} \\
\cmidrule(lr){2-3} \cmidrule(lr){4-5} 
d /  & $C_{\sigma}$ &  $C_{\beta}$  & $C_{\sigma}$ &  $C_{\beta}$  \\
\specialrule{1.5pt}{0pt}{0pt}
1 & 0.30 & 1.00   & 0.10 & 1.00 \\
3  & 0.30 & 1.50   & 0.10 & 1.50  \\
11  & 0.30 & 2.00    & 0.10 & 2.00 \\
\bottomrule
\end{tabular}
\label{table:Parameter_Choices_posteriori_convergence_rates}
\end{table}

\begin{table}[htbp]
\caption{ Choices of $\{C_\sigma,C_\beta\}$.}
\centering
\begin{tabular}{ccccc} 
\toprule
  & \multicolumn{2}{c}{The empirical $L^{1}$-case} & \multicolumn{2}{c}{The empirical $L^{\infty}$-case} \\
\cmidrule(lr){2-3} \cmidrule(lr){4-5} 
d/ & $C_{\sigma}$ &  $C_{\beta}$  & $C_{\sigma}$ &  $C_{\beta}$  \\
\specialrule{1.5pt}{0pt}{0pt}
1 & 0.20 & 1.50    & 0.20 & 1.00 \\
3  & 0.20 & 1.00& 0.10 & 2.00  \\
11  & 0.30 & 2.00 & 0.30 & 2.00 \\
\bottomrule
\end{tabular}
\label{table:Parameter_Choices_Probability simulations}
\end{table}

\end{document}